\documentclass[12pt,oneside]{amsart}

\usepackage{amssymb,latexsym}

\usepackage{enumerate}

\usepackage{hyperref}
\hypersetup{
  colorlinks   = true, 
  urlcolor     = blue, 
  linkcolor    = blue, 
  citecolor   = red 
}
\usepackage{anysize}
\marginsize{2cm}{2cm}{2cm}{2cm}

\newtheorem{theorem}{Theorem}[section]

\newtheorem{lemma}[theorem]{Lemma}

\newtheorem{conjecture}[theorem]{Conjecture}

\theoremstyle{definition}

\theoremstyle{remark}
\newtheorem{remark}[theorem]{Remark}

\numberwithin{equation}{section}

\DeclareMathOperator{\vol}{vol}
\DeclareMathOperator{\area}{area}
\DeclareMathOperator{\per}{per}

\title{Bounding minimal solid angles of polytopes}

\author{Arseniy~Akopyan{$^\spadesuit$}}

\email{akopjan@gmail.com}

\author{Roman~Karasev{$^\clubsuit$}}

\email{r\_n\_karasev@mail.ru}
\urladdr{http://www.rkarasev.ru/en/}

\thanks{$^\spadesuit$ Supported by People Programme (Marie Curie Actions) of the European Union's Seventh Framework Programme (FP7/2007-2013) under REA grant agreement n$^\circ$[291734].}

\thanks{{$^\clubsuit$} Supported by the Russian Foundation for Basic Research grant 15-31-20403 (mol\_a\_ved).}
\thanks{{$^\clubsuit$} Supported by the Russian Foundation for Basic Research grant 15-01-99563 A}

\address{{$^\spadesuit$} Institute of Science and Technology Austria (IST Austria), Am Campus 1, 3400 Klosterneuburg, Austria}
\address{{$^\clubsuit$} Moscow Institute of Physics and Technology, Institutskiy per. 9, Dolgoprudny, Russia 141700}
\address{{$^\clubsuit$} Institute for Information Transmission Problems RAS, Bolshoy Karetny per. 19, Moscow, Russia 127994}

\begin{document}
	
\maketitle

\begin{abstract}
In this article we study the following question: What can be the measure of the minimal solid angle of a simplex in $\mathbb{R}^d$? We show that in dimensions three it is not greater than the solid angle of the regular simplex. And in dimension four the same holds for simplices sufficiently close to the regular simplex. We also study a similar question for trihedral and dihedral angles of polyhedra compared to those of regular solids.
\end{abstract}

\section{Introduction}

In~\cite[Question 7]{karasev2015bounds} it was conjectured that every simplex in $\mathbb R^d$ has a solid angle at a vertex, whose $(d-1)$-dimensional spherical measure is at most that of the solid angle of a regular simplex. In this note we confirm this conjecture for dimension $d=3$ and show that the $4$-dimensional regular simplex is a local extremum for the problem. Of course, for $d=2$ the question is trivial.

We start with a couple of general observations. It is impossible to bound the maximal solid angle of a simplex from below, which would be the opposite version of the conjecture. The corresponding example is an ``almost flat'' simplex, which can be described in the following way: Take $d+1$ points in $\mathbb R^{d-1}$ in convex position and then slightly perturb them in $\mathbb R^d$ so that they make a non-degenerate simplex. All the solid angles of such a simplex are close to zero. 

Another example shows that the problem cannot be approached by an averaging argument: A tetrahedron whose base is a triangle and whose remaining vertex is very close to the mass center of its base has one solid angle close to $2\pi$ and other three solid angles close to $0$. So the average is close to $\pi/2$, which is much more than the solid angle of the regular tetrahedron. Similar examples exist in all dimensions.

The similar problem for estimating the minimal dihedral angle of a simplex from above is more or less equivalent to the Jung theorem about covering a set of diameter $1$ by a minimal possible ball. See~\cite{maehara2013dihedral} for a nice proof.

\subsection*{Acknowledgments.}
We would like to thank Ivan Izmestiev for pointing our attention to Formula \ref{eq:three-crofton}, Alexey Balitskiy, Senya Shlosman, Oleg Ogievetsky and Endre Makai for fruitful discussions.

The first version of the paper contained wrong statements about global minimality in the four dimensional case. We are grateful to the anonymous reviewers from the journal Discrete \& Computational Geometry, who pointed out that B\"{o}hm only proved the local minimality of the regular tetrahedron in the isoperimetric problem for spherical tetrahedra. 

\section{Minimal solid angle of simplex: The general approach}

Now let us describe our approach to the problem.

\begin{conjecture}
	\label{con:main conjecture}
Any $d$-dimensional simplex $T\subset\mathbb{R}^d$, has a solid angle not greater than the solid angle of the $d$-dimensional regular simplex.
\end{conjecture}

We will show that for any $d$ Conjecture \ref{con:main conjecture} follows from the following:

\begin{conjecture}
	\label{con:isodual inequality}
	For $(d-1)$-dimensional spherical simplices $S \subset \mathbb{S}^{d-1}$, the maximal $\vol S$ under given $\vol S^\circ$ is attained at the regular simplex.
\end{conjecture}

\begin{proof}[Reduction of Conjecture \ref{con:main conjecture} to Conjecture~\ref{con:isodual inequality}]
In what follows we associate convex cones in $\mathbb R^d$ with convex subsets of the sphere $\mathbb S^{d-1}$. So the solid angle at a vertex of the simplex is a subset of $\mathbb S^{d-1}$ in our argument. Also, for a convex closed $X\subset\mathbb S^{d-1}$, we consider its polar $X^\circ$ as a subset of the sphere $\mathbb S^{d-1}$.

Now let $V_0$, $V_1$, \dots, $V_{d}$ be the $d$-dimensional spherical simplices corresponding to solid angles of the vertices of $T$. Let $V_0^\circ$, $V_1^\circ$, \dots, $V_{d}^\circ$ be their polar simplices. Evidently, the spherical simplices (Euclidean cones) $V_i^\circ$ constitute the normal fan of $T$ and therefore provide a partition of $\mathbb S^{d-1}$.

Let $C_{d-1}$ be the $(d-1)$-dimensional spherical simplex corresponding to the solid angle of the regular simplex in $\mathbb{R}^{d}$, like one of the above described $V_i$'s. It is clear that $C^\circ_{d-1}$ is the regular simplex from the standard partition of the $(d-1)$-dimensional sphere into $d+1$ simplices.

Assume Conjecture \ref{con:isodual inequality} is valid. Choosing $i$ with maximal $\vol (V_i)^\circ$ we obtain that its volume is not less than $\vol (C_{d-1}^\circ)$, therefore $\vol (V_i) \leq \vol (C_{d-1})$.
\end{proof}

In the next sections we prove Conjecture \ref{con:isodual inequality} for the case $d=3$ and prove a weaker claim, the local extremality of the regular simplex, for $d=4$.

The proof will be based on the isoperimetric inequality for spherical simplices, which is only solved for the case of the two-dimensional sphere. For the three-dimensional sphere, J.~B\"{o}hm~\cite{bohm1986zumisoperimetrischen} proved the local extremality of the regular simplex, that is why our result is local for $d=4$.

\begin{conjecture}[Isoperimetric inequality for a spherical simplex]
\label{con:isoperemetric simplex}
Out of $(d-1)$-dimensional spherical simplices with fixed volume, the regular simplex has the minimal surface $(d-1)$-volume.
\end{conjecture}

We should mention that if Conjectures~\ref{con:isodual inequality} and \ref{con:isoperemetric simplex} both are true than the statement of Conjectures~\ref{con:main conjecture} could be generalized:

\begin{conjecture}[Generalization of Conjecture~\ref{con:main conjecture}]
	\label{con:main conjecture generalized}
Any $d$-dimensional simplex $T\subset\mathbb{R}^{d}$ for any $k$, $k<d-1$ has $k$-face with the $(d-k-1)$-dimensional solid angle corresponding to this $k$-face not greater than the $(d-k-1)$-dimensional solid angle of the $k$-face of $d$-dimensional regular simplex.
\end{conjecture}

Indeed, as in the reduction above, one can consecutively apply Conjecture~\ref{con:isoperemetric simplex} to maximal faces of simplex $\vol (V_i)^\circ$ of maximal volume and obtain that one of its $k$-coface has volume not less than the volume of the $k$-coface of $\vol (C_{d-1}^\circ)$. Note that this $k$-coface is polar to the solid angle of the corresponding $k$-face of $T$. Then applying Conjecture~\ref{con:isodual inequality} we obtain that the corresponding $k$-face has the $(d-k-1)$-dimensional solid angle less then the $(d-k-1)$-dimensional solid angle of a $k$-face of the regular simplex.

\section{The three-dimensional case}

We show that for $d=3$ all four Conjectures~\ref{con:main conjecture}-\ref{con:main conjecture generalized} hold. 

The following lemma is well-known and it is the two-dimensional version of Conjecture~\ref{con:isoperemetric simplex}:

\begin{lemma}[Isoperimetric inequality for a spherical triangle, \cite{toth1965regular,bohm1986zumisoperimetrischen}]
\label{lem:bohm triangle}
Out of spherical triangles with fixed area, the regular triangle has the minimal perimeter.
\end{lemma}

From the spherical triangle formula, for any triangle $\triangle \in \mathbb{S}^2$ and its polar $\triangle^\circ$ the following equation follows immediately:
\begin{equation}
	\label{eq:two-crofton}
	\area \triangle + \per (\triangle ^\circ)=2\pi.
\end{equation}
It seems that this formula was well known before, for example L. Fejes T\'{o}th using \eqref{eq:two-crofton} proves existence of the Lexell circle \cite[\S 1.8]{toth1972lagerungen}.

From this formula, in particular, it follows that $\area (C)+\per(C^\circ)=2\pi$, for the regular triangle $C$ with $\area (C^\circ)=\area (S^\circ)$.
Now from Lemma~\ref{lem:bohm triangle} follows that the perimeter of $S^\circ$ is not less than the perimeter of $C^\circ$. Using~\eqref{eq:two-crofton} we obtain that $\area (S) \geq \area (C)$.
That is the two-dimensional version of Conjecture~\ref{con:isodual inequality}, and Conjectures~\ref{con:main conjecture} and \ref{con:main conjecture generalized} for $d=3$ follow as corollaries.

It is clear that the scheme of this proof works for platonic solids as well; let us state the corresponding results.

\begin{theorem}
	\label{thm:polytopes in r3}
Let $P$ be a polyhedron (non necessarily convex) in $\mathbb R^3$, combinatorially equivalent to a platonic solid $P_0$. Then 
\begin{enumerate}[(i)]
	\item \label{item: solids of polytopes in r3}
	one of its solid angles is not greater than the solid angle of the platonic solid $P_0$;
	\item \label{item: dihedral of polytopes in r3}
	one of its dihedral angles is not greater than the dihedral angle of $P_0$.
\end{enumerate}

For a simple polyhedron $P$, the combinatorial equivalence to $P_0$ is not needed, it suffices to assume that $P$ and $P_0$ have the same number of vertices.
\end{theorem}

The statement \eqref{item: dihedral of polytopes in r3} follows from \eqref{item: solids of polytopes in r3}. Indeed, the minimal dihedral angle of those adjacent to the minimal solid angle is not greater than the dihedral angle of $P_0$, because otherwise, by the spherical area formula, the solid angle would not be minimal.

For simple platonic solids (the cube and the dodecahedron), the proof of \eqref{item: solids of polytopes in r3} is absolutely the same as the proof of the case of the tetrahedron. For non-simple case (the octahedron and the icosahedron), it is needed to the generalize the inequality between the area of a triangle and the area of its dual as follows:

\begin{theorem}
\label{thm:isoperimetric polygon}
For a spherical $n$-gon $P$, the maximal $\area(P)$ under given $\area(P^\circ)$ is attained at the regular $n$-gon.
\end{theorem}

The scheme of the proof is the same as for Conjecture~\ref{con:isodual inequality} for the case $d=3$, we use the generalization of Lemma~\ref{lem:bohm triangle} for any $n$-gone:

\begin{lemma}[Spherical isoperimetric inequality for $n$-gons, L.~Fejes T{\'o}th {\cite[Section 30]{toth1965regular}}]
\label{lem:isoperimetric polyton}
Out of spherical $n$-gons with fixed perimeter, the regular $n$-gon has the maximal area.
\end{lemma}

For the case hyperbolic plane Lemma~\ref{lem:isoperimetric polyton} was proved in K.~Bezdek~\cite{bezdek1984einelementarer}. See also \cite{csikos2006generalization} for a generalization of these statements to the class of polygons formed by circular arcs of fixed radius.

%
%
%
%
%
%
%

\section{The four-dimensional case}

Now we utilize the result of J.~B\"{o}hm:

\begin{lemma}[Isoperimetric inequality for a spherical tetrahedron, \cite{bohm1986zumisoperimetrischen}]
\label{lem:bohm tetrahedron}
Out of spherical tetrahedra with fixed volume, the regular tetrahedron has the \emph{local} minimal surface area.
\end{lemma}

For the regular spherical tetrahedron $C$, denote by $U(C)$ its neighborhood in the space of all spherical tetrahedra with volume equal to the volume of $C$ where $C$ locally minimize the surface area. Let the neighborhood $U^\circ(C)$ consist of the polars to the simplices from $U(C^\circ)$. 

Let the set $\mathcal U$ be the union of $\cup (U(C) \cap U^\circ (C))$ over all regular simplices $C$.

Suppose $S$ belongs to $\mathcal{U}$ and let $C$ be the regular spherical tetrahedron with $\vol(C^\circ)=\vol(S^\circ)$, then $S \in U(C)$, $S^\circ \in U(C^\circ)$.

From Lemma~\ref{lem:bohm tetrahedron} it follows that $\area (\partial S^\circ)\geq \area (\partial C^\circ)$. Now apply the following equality (see \cite[Proof of Theorem 4.2]{izmestiev2014infinitesimal}, this equality can be proven by direct application of the formula for the area of a spherical triangle):
\begin{equation}
\label{eq:three-crofton}
\area (\partial X) +  \area(\partial X^\circ)=4\pi.
\end{equation}
	
This implies 
\[
\area (\partial S) \leq \area (\partial C).
\]
It remains to apply Lemma~\ref{lem:bohm tetrahedron} once again and obtain $\vol (S)\leq \vol (C)$.

This proves Conjecture~\ref{con:isodual inequality} for $S \in \mathcal{U}$. And therefore Conjectures~\ref{con:main conjecture} and \ref{con:main conjecture generalized} hold for the class of four-dimensional simplices whose solid angles belong to $\mathcal{U}$.

Let us summarize everything as a theorem:
\begin{theorem}
\label{thm:tetrahedral}
The four-dimensional regular simplex $T$ has a neighborhood $U(T)$ (in space of four-dimensional simplex) such that any simplex from $U(T)$
\begin{enumerate}[(i)]
	\item \label{thm: tetrahedral: solid angle} has a solid angle at some vertex not greater than the solid angle at a vertex of the regular simplex;
	\item \label{thm: tetrahedral: trihedral angle} has an edge with the trihedral angle not greater than the trihedral angle at an edge of the regular simplex;
	\item \label{thm: tetrahedral: diheral angle} has a two-face with the dihedral angle not greater than the dihedral angle at a two-face of the regular simplex.
\end{enumerate}
 \end{theorem}

%
%
%
%

For dimension three in spherical geometry, we do not know how to prove even the local analogue of Lemma~\ref{lem:isoperimetric polyton} for other platonic solids. Therefore we only formulate an analogue of the simplex result for the tesseract ($4$-cube) and the hecatonicosachoron ($120$-cell):

\begin{theorem}
\label{thm:4-cube}
The tesseract $T$ has a neighborhood $U(T)$ (in space of simple polytopes with $16$ vertices) such that any polytope from $U(T)$
\begin{enumerate}[(i)]
	\item \label{thm: 4-cube: solid angle} has a solid angle at some vertex not greater than the solid angle at a vertex of the tesseract;
	\item \label{thm: 4-cube: trihedral angle} has an edge with the trihedral angle not greater than the trihedral angle at an edge of the tesseract;
	\item \label{thm: 4-cube: diheral angle} has a two-face with the dihedral angle not greater than the dihedral angle at a two-face of the tesseract.
\end{enumerate}
 \end{theorem}
 
\begin{theorem}
\label{thm:120-cell}
The the hecatonicosachoron $H$ has a neighborhood $U(H)$ (in space of simple polytopes with $600$ vertices) such that any polytope from $U(H)$

\begin{enumerate}[(i)]
	\item has a solid angle at some vertex not greater than the solid angle at a vertex of the hecatonicosachoron;
	\item has an edge with the trihedral angle not greater than the trihedral angle of the hecatonicosachoron;
	\item has a two-face with the dihedral angle not greater than the dihedral angle at a two-face of the hecatonicosachoron.
\end{enumerate}
\end{theorem}

\begin{proof}[Sketch of the proof]
Conclusion~\eqref{thm: 4-cube: solid angle} is proved similarly to Theorem~\ref{thm:tetrahedral} for the case of the four-dimensional simplex.
	
Conclusion~\eqref{thm: 4-cube: trihedral angle} is proved similarly to Theorem~\ref{thm:tetrahedral}.
	
Conclusion~\eqref{thm: 4-cube: diheral angle} is proved similarly to the proof of \eqref{item: dihedral of polytopes in r3} of Theorem~\ref{thm:polytopes in r3}: We choose the minimal trihedral angle and note that one of its dihedral angles should be not greater than a dihedral angle of the corresponding regular polytope (tesseract or hecatonicosachoron), since this trihedral angle is itself not greater than the trihedral angle of the regular polytope.
\end{proof}

\begin{remark}
Like in Theorem~\ref{thm:polytopes in r3}, we can drop the convexity condition on $P$ and assume only that it is a $PL$-sphere on $16$ or $600$ vertices, and each vertex of $P$ is simple.
\end{remark}


\section{Other remarks and questions}

1.  Returning to the general case of~\cite[Question 7]{karasev2015bounds}, we think that Conjecture~\ref{con:main conjecture} is valid in all dimensions. This must be clear from the above proofs that the validity of this conjecture is sufficient to prove Conjecture~\ref{con:isodual inequality} in all dimensions. Note that a similar result was established in~\cite{gao2003intrinsic}: For all spherical convex sets, the maximal $\vol (X)$ under given $\vol (X^\circ)$ is attained at a spherical cap. In  \cite{karasev2015bounds} this was used to give a good asymptotic bound for the maximal solid angle.

2. The proof of the two and three dimensional versions of Conjecture~\ref{con:isodual inequality} uses equations \eqref{eq:two-crofton} and \eqref{eq:three-crofton} connecting the area (or the perimeter) of a triangle with its dual. These relations are not a coincidence.

The Crofton formula for a convex subset $X\subset\mathbb S^d$ states that:
\[
\vol_{d-1} \partial X = c_{d, 1} \mu_{d, 1} \{\ell : \ell\cap X\neq\emptyset\},
\]
where $\ell$ is a spherical line in $\mathbb S^d$, $\mu_{d, 1}$ is an $\mathrm{SO}(d+1)$-invariant measure on the set of such lines, and $c_{d,1}$ is a certain coefficient, whose value we need not know.

Note that a point $p$ or $-p$ belongs to a given convex set $X\subset \mathbb{S}^2$ if and only if their polar line $p^\circ$ does not intersect the interior of the polar $X^\circ\subset \mathbb{S}^2$, this is evident from the version of the Hahn--Banach theorem for cones. From this observation and the Crofton formula (after adjusting the constants) follows \eqref{eq:two-crofton}: $\area X + \per (X^\circ)=2\pi.$

For the three-dimensional sphere we make the following observation: The polarity in $\mathbb{S}^3$ maps lines to lines and preserves the $\mathrm{SO}(4)$-invariant measure. A line $\ell$ intersects $X$ if and only if the line $\ell^\circ$ does not intersect $X^\circ$. This gives \eqref{eq:three-crofton}: $\area (\partial X) +  \area(\partial X^\circ)=4\pi.$

By a similar argument one can prove analogous statements for intrinsic volumes of spherical bodies of higher dimension.

3. As we mention in the introduction, all solid angles of some simplices can be close to zero. Hence we cannot find a prescribed inequality in the sets of solid angles of any two given simplices (except for the case $d=2$).

On the other hand, for any two given simplices $T_1$ and $T_2$ (with numbered vertices) in $\mathbb{R}^d$ one can always find two corresponding $(d-2)$-faces $t_1$ and $t_2$ such that the dihedral angle of $t_1$ is not greater than the dihedral angle of $t_2$, \cite{rivin1998similarity}. In view of this we can ask the following:

\begin{conjecture}
Let $P_1$ and $P_2$ be two combinatorially equivalent convex polytopes in $\mathbb{R}^d$. Then there exist corresponding $(d-2)$-faces $t_1$ of $P_1$ and $t_2$ of $P_2$ such that the dihedral angle of $t_1$ is not greater than the dihedral angle of $t_2$.
\end{conjecture}

\bibliographystyle{abbrv}
\bibliography{solidangles}{}

\end{document}